\documentclass[11pt]{amsart}
\usepackage{latexsym,amssymb,amsmath,youngtab}
\usepackage{youngtab}
\usepackage{epsfig}
\usepackage{amsmath,amsthm,amssymb,amscd,MnSymbol}
\usepackage[mathscr]{eucal}
\usepackage{verbatim, hyperref}
\usepackage{color}
\newcommand*{\textlabel}[2]{%
  \edef\@currentlabel{#1}
  \phantomsection
  #1\label{#2}
}

\newtheoremstyle{custom}
  {3pt}
  {3pt}
  {\slshape}
  {}
  {\bfseries}
  {.}
  { }
   {}
\theoremstyle{custom}
\newtheorem{theorem}{Theorem}[section]
\newtheorem{proposition}[theorem]{Proposition}

\newtheorem{proposition/definition}[theorem]{Proposition/Definition}
\newtheorem{lemma}[theorem]{Lemma}

\theoremstyle{definition}
\newtheorem{definition}[theorem]{Definition}

\newtheorem{example}[theorem]{Example}

\theoremstyle{remark}
\newtheorem{remark}[theorem]{Remark}

\newtheoremstyle{exercise}
  {3pt}
  {6pt}
  {}
  {}
  {\bfseries}
  {:}
  { }
   {}
\theoremstyle{exercise}
\newtheorem{exercise}[theorem]{Exercise}
\newtheoremstyle{exercises}
  {3pt}
  {6pt}
  {}
  {}
  {\bfseries}
  {:}
  {\newline}
   {}

\theoremstyle{exercise}
\newtheorem{exercises}[theorem]{Exercises}

\input epsf
\def\boxit#1{\vbox{\hrule height1pt\hbox{\vrule width1pt\kern3pt
  \vbox{\kern3pt#1\kern3pt}\kern3pt\vrule width1pt}\hrule height1pt}}

\def\11{\mathbf 1}

\def\dim{{\rm dim}\;}

\def\be{\begin{equation}}
\def\ene{\end{equation}}

\def\im{{\rm Im}}
\def\dim{{\rm dim }}
\begin{document}

\title{Brill's equations as a $GL(V)$-module}
\author{Yonghui Guan}
 \begin{abstract}
 The Chow variety of polynomials that decompose as a product of linear forms has been studied for more than 100 years. Brill, Gordon \cite{Gordon} and others obtained set-theoretical equations for the Chow variety. In this article, I compute Brill's equations as a $GL(V)$-module.
 \end{abstract}
 \email{yonghuig@math.tamu.edu }
\maketitle
\section{Introduction}
\subsection{Motivation}
There has been substantial recent interest in the equations of certain algebraic varieties that encode
natural properties of polynomials (see e.g. \cite{BBeis,MR3081636,MR3169697,BGLeis,LWsecseg,LMsec,MR2310544}). Such varieties are usually preserved by algebraic
groups and it is a natural question to understand the module structures of the spaces of equations. One variety of interest
is the {\it Chow variety} of polynomials that decompose as a product of linear forms, which is defined by
$Ch_d(V)=\mathbb{P}\{z\in S^dV|z=w_1\cdots w_d\,{\rm\ for\ some\ } w_i \in V\}\subset\mathbb{P}S^dV,$
where $V$ be a finite-dimensional complex vector space and  $\mathbb{P}S^dV$ is the projective space of homogeneous polynomials of degree $d$ on the dual space $V^*$.

The ideal of the Chow variety of polynomials that decompose as a product of linear forms has been studied for over 100 years, dating back at least to  Gordon and Hadamard. Let $S^\delta(S^dV)$ denote the space of homogeneous polynomials of degree $\delta$ on $S^dV^*$. The {\it Foulkes-Howe} map $h_{\delta,d}:S^\delta(S^dV)\rightarrow S^d(S^\delta V)$ was defined by Hermite \cite{hermite} when $\dim\ V=2$, and Hermite proved the map is an isomorphism in his celebrated \lq\lq Hermite reciprocity\rq\rq. Hadamard \cite{MR1554881} defined the map in
general and observed that its kernel is $I_\delta(Ch_d(V^*))$, the degree $\delta$ component of the ideal
of the Chow variety. We do not understand this map when $d>4$ (see \cite{MR1504330,MR983608,MR0037276,MR1651092,MR1243152,MR1601139}).

Brill and Gordon (see \cite{gkz,Gordon,MR2865915}) wrote down set-theoretic equations
for the Chow variety of degree $d+1$,  called \lq\lq Brill's equations\rq\rq. Brill's equations give a geometric derivation of set-theoretic equations for the Chow variety, it is a natural question to understand these equations in terms of a $GL(V)$-module from a representation-theoretic perspective, where $GL(V)$  denotes the {\it General Linear Group} of invertible linear maps from $V$ to $V$.


\subsection{Result}
The group $GL(V)$ has an induced action on  $S^dV$ (see \S\ref{Gvariety}). The Chow variety $Ch_d(V)$ is invariant under the action of $GL(V)$, therefore the ideal of $Ch_d(V)$ is a $GL(V)$-module (see \S\ref{Gvariety}). For any partition $\lambda$, let $S_\lambda V$ be the irreducible $GL(V)$-module determined by the partition $\lambda$. For example $S_{(d)}V = S^dV$, while $S_{(1^d)}V = \Lambda^dV$ is
the $d$-th exterior power of $V$.

\begin{theorem}\label{Brillmodule}
Assume $\dim$ $V\geq3$ and $d\geq2$. The degree $d+1$ equations for $Ch_d(V)$ discovered by Brill, as a $GL(V)$-module, are:
\begin{eqnarray*}
\begin{cases}
S_{(7,3,2)}V^* \ \   d=3;\\
\bigoplus_{j=2}^{d}S_{(d^2-j,d,j)}V^*\  d\neq3.
\end{cases}
\end{eqnarray*}
\end{theorem}
\begin{remark}
Compare the codimension of $Ch_d(\mathbb{C}^3)$ with the dimension of all the modules in Theorem \ref{Brillmap} that define $Ch_d(\mathbb{C}^3)$ set-
theoretically: When $d=2$, the codimension of $Ch_2(\mathbb{C}^3)$ is $1$ and the dimension of $S_{(2,2,2)}\mathbb{C}^{*3}$ is $1$. When $d=3$, the codimension of $Ch_3(\mathbb{C}^3)$ is $3$ while the dimension of $S_{(7,3,2)}\mathbb{C}^{*3}$ is $35$. In general the dominant term of the codimension of $Ch_d(\mathbb{C}^3)$
is $\frac{d^2}{2}$ , but the dominant term of the dimension of  the modules from Brill's equations that define $Ch_d(\mathbb{C}^3)$ is $\frac{d^7}{2}$. Therefore, the Chow variety
is far from being a complete intersection.
\end{remark}
\subsection{Overview of method}

Brill's equations \cite{gkz,Gordon,MR2865915} are set-theoretical equations for the Chow variety $Ch_d(V)$.  The Chow variety $Ch_d(V)$
is the zero set of a polynomial map $\mathfrak{B}:S^dV\rightarrow S_{(d,d)}V\otimes S^{d^2-d}V$ of degree $d+1$ (see $\S \ref{polymap}$).
I determine Brill's equations as a $GL(V)$-module to understand these equations and write down these equations explicitly.

The idea is to construct the polarization (see \S\ref{polarization}) $\bar{\mathfrak{B}}$ of $\mathfrak{B}$, where
$\bar{\mathfrak{B}}:S^{d+1}(S^dV)\rightarrow S_{(d,d)}V\otimes S^{d^2-d}V$, and then determine the image of $\bar{\mathfrak{B}}$, whose dual is isomorphic to the $GL(V)$-module corresponding to Brill's equations. I call $\bar{\mathfrak{B}}$ Brill's map.

Brill's map $\bar{\mathfrak{B}}$ is a $GL(V)$-module map,
the space $S_{(d,d)}V\otimes S^{d^2-d}V$ can be decomposed by Pieri's rule (see e.g. \cite{FH}),
\begin{eqnarray*}
S_{(d,d)}V\otimes S^{d^2-d}V=\bigoplus_{j=0}^{d}S_{(d^2-j,d,j)}V,
\end{eqnarray*}
I determine which irreducible $GL(V)$-modules are in the image of Brill's map.
\subsection{Organization}
In $\S\ref{Background}$ I define Brill's equations $\mathfrak{B}$ following the notation in \cite{MR2865915} in $\S 8.6$,
and  review the polarization of a polynomial map, G-variety and how to write down highest weight vectors of a module via raising operators. In $\S\ref{imageBrillmap}$, I use the polarization of Brill's polynomial map $\mathfrak{B}$ to construct Brill's map $\overline{\mathfrak{B}}$, which is a $GL(V)-$module map. I write down the highest vectors of the modules $S_{(d^2-j,d,j)}V\subset S_{(d,d)}V\otimes S^{d^2-d}V$, then I compute the images of some given weight vectors $v_j$ in $S^{d+1}(S^dV)$ under Brill's map $\overline{\mathfrak{B}}$, and determine whether the projection of $\overline{\mathfrak{B}}(v_j)$ to the module $S_{(d^2-j,d,j)}V$ is zero to determine the image of Brill's  map.
\subsection{Acknowledgement}
 I thank my advisor J.M. Landsberg for discussing all the details throughout this article. I thank C. Ikenmeyer  for discussing the ideal of
 Chow variety.
\section{Preliminaries }\label{Background}
\subsection{Polarization of a polynomial map}\label{polarization}
\begin{definition}
Let $V_1,\cdots, V_d$ be complex vector spaces, define a map $\varphi: V_1\times\cdots\times V_d\rightarrow V_1\otimes\cdots\otimes V_d$
by $\varphi(v_1,\cdots,v_d)=v_1\otimes\cdots\otimes v_d$. {\it The universal property of tensors}:
given a complex vector space $W$ and a multi-linear map $h:V_1\times\cdots\times V_d\rightarrow W$, there is a unique
 linear map $\tilde{h}:V_1\otimes\cdots\otimes V_d\rightarrow W$, such that $h=\tilde{h} \circ\varphi$.
\end{definition}
\begin{definition}
Let $W$ be a complex vector space, a map $P: W\rightarrow \mathbb{C}^m$
is a polynomial map of degree $k$ if
$P=(P_1,\cdots,P_m),$
and each $P_i$ $(i=1,\cdots,m)$ is a homogenous polynomial of degree $k$ on $W$.\\
Define the {\it complete\ polarization}
$\bar{P}:W\times\cdots\times W\rightarrow \mathbb{C}^m$
of $P$ to be
\begin{eqnarray*}
\bar{P}(w_1,\cdots,w_k)=\frac{1}{k!}\sum_{I\subset[k],I\neq\emptyset}(-1)^{k-|I|}P(\sum_{i\in I}w_i).
\end{eqnarray*}
Where $[k]=\{1,\cdots,k\}$, $w_i\in W$ and $\bar{P}$ is a symmetric multi-linear map.
By the universal property of tensors, $\bar{P}$ is considered as a map
$\bar{P}:W^{\otimes k}\rightarrow\mathbb{C}^m$.
By the symmetry of $\bar{P}$, $\bar{P}$ can be also seen as a map
$\bar{P}:S^kW\rightarrow \mathbb{C}^m$,
such that
 \begin{eqnarray}\label{por}
\bar{P}(w_1\cdots w_k)=\frac{1}{k!}\sum_{I\subset[k],I\neq\emptyset}(-1)^{k-|I|}P(\sum_{i\in I}w_i),
\end{eqnarray}
and it can be extended linearly to the whole space.
\end{definition}
\begin{example}
Let {\rm dim} $V$=2, and let $\{e_1,e_2\}$ be a basis of $V$. Consider the polynomial map
$P:V\rightarrow \mathbb{C}^2$ defined by
$$a_1e_1+a_2e_2\mapsto (a_1^2,a_1^2+a_2^2).$$
$P$ is a polynomial map of degree 2, so by \eqref{por} $\bar{P}:S^2V\rightarrow \mathbb{C}^2$ is defined by
\begin{eqnarray*}
 \bar{P}((a_1e_1+a_2e_2)(a_3e_1+a_4e_2))&=&\frac{1}{2}[P(a_1e_1+a_2e_2+a_3e_1+a_4e_2)\\
 &&-P(a_1e_1+a_2e_2)-P(a_3e_1+a_4e_2)]\\
 &=&\frac{1}{2}[((a_1+a_3)^2,(a_1+a_3)^2+(a_2+a_4)^2)\\
 &&-(a_1^2,a_1^2+a_2^2)-(a_3^2,a_3^2+a_4^2)]\\
 &=&(a_1a_3,a_1a_3+a_2a_4).
 \end{eqnarray*}
 Therefore
 \begin{eqnarray*}
 \bar{P}(a e_1^2+be_1e_2+ce_2^2)&=&a\bar{P}(e_1^2)+b\bar{P}(e_1e_2)+c\bar{P}(e_2^2)\\
 &=&(a,a)+(0,0)+(0,c)\\
 &=&(a,a+c).
 \end{eqnarray*}
\end{example}

\subsection{G-variety}\label{Gvariety}
I follow the notation in \cite{MR2865915} in $\S 4.7$.
\begin{definition}
Let $W$ be a complex vector space. A variety $X\subset\mathbb{P}W$ is called a {\it G-variety} if $W$ is a module for the group  $G$
and for all $g\in G$ and $x\in X$, $g\cdot x\in X.$
\end{definition}
G has an induced action on $S^dW^*$ such that for any $P\in S^dW^*$ and $w\in W$,
$g\cdot P(w)=P(g^{-1}\cdot w)$. $I_d(X)$ is a linear subspace of $S^dW^*$ that is invariant
under the action of $G$, therefore:
\begin{proposition}
If $X\subset\mathbb{P}W$ is a $G$-variety, then the ideal of $X$ is a $G$-submodule of $S^\bullet W^*:=\oplus_{d=0}^{\infty}S^dW^*$.
\end{proposition}
\begin{example}
The Group $GL(V)$ has an induced action on $S^dV$ and $S^k(S^dV^*)$ similarly. The Chow variety $Ch_d(V)$
 is invariant under the action of $GL(V)$, therefore it is a $GL(V)$-variety and its ideals is $GL(V)$-submodules of $S^\bullet (S^dV^*)=\oplus_{k=0}^{\infty}S^k(S^dV^*)$.
\end{example}

Let $X\subset\mathbb{P}W$ be a $G$-variety, and $M$ be an irreducible submodule of $S^\bullet W^*$, then either $M\subset I(X)$
or $M\cap I(X)=\emptyset$. Thus to test if $M$ gives equations for $X$, one only need to test one polynomial in $M$.
\subsection{Representation theory}
I follow the notation in \cite{FH}.
Let dim $V=n$ and  $\{e_1,e_2,\cdots,e_n\}$ be a basis of $V$. The group $GL(V)$ has a natural action on $V^{\otimes d}$ such that
$g\cdot(v_1\otimes v_2\cdots\otimes v_d)=g\cdot v_1\otimes\cdots\otimes g\cdot v_d$.
Let $B\subset GL(V)$ be the subgroup of upper-triangular matrices (a {\it Borel subgroup}).
For any partition $\lambda=(\lambda_1,\cdots,\lambda_n)$ with order $d$,  there is a unique line in $S_\lambda V\subset V^{\otimes d}$ that is preserved by $B$, called a {\it highest weight line}.
Let $\mathfrak{gl}(V)$ be the Lie algebra of $GL(V)$,
there is an induced action of $\mathfrak{gl}(V)$ on $V^{\otimes d}$. For $X\in \mathfrak{gl}(V)$,
$$X.(v_1\otimes v_2\cdots\otimes v_d)=X.v_1\otimes v_2\cdots\otimes v_d+v_1\otimes X.v_2\otimes\cdots \otimes v_d+\cdots+v_1\otimes v_2\cdots \otimes v_{d-1}\otimes X.v_d.$$
Let $E^i_j\in\mathfrak{gl}(V)$ such that $E^i_j(e_j)=e_i$ and $E^i_j(e_k)=0$ when $k\neq j$. If $i<j$,  $E^i_j$ is called
 a {\it raising operator}; if  $i>j$,  $E^i_j$  is called a {\it lowering operator}.

A {\it highest weight vector} of a $GL(V)$-module is a weight vector that is killed by all raising operators.
Each realization of the module $S_\lambda V$ has a unique highest weight line. Let $W$ be a $GL(V$)-module, the multiplicity of $S_\lambda V$ in $W$ is equal to the dimension of the highest weight space with respect to the partition $\lambda$.

Define the weight space $W_{(a_1,\cdots,a_n)}$$\subset$ $ S^k(S^dV)$ to be the set of all the weight vectors whose weights are $(a_1,\cdots,a_n)$.
Note that  $S^dV$ has a natural basis $\{e_1^{\alpha_1}\cdots e_n^{\alpha_n}\}_{\alpha_1+\cdots+\alpha_n=d} $.
\begin{example}
$S_{(4,2)}V\subset S^3(S^2V)$ has multiplicity 1.
\begin{proof}
Let $v$ be a highest weight vector of $S_{(4,2)}V$. The weight space $W_{(4,2)}$ has a basis
$\{(e_1^2)^2(e_2^2),(e_1^2)(e_1e_2)^2\}$. Write
$v=a(e_1^2)^2(e_2^2)+b(e_1^2)(e_1e_2)^2$, then $E^1_2v=0$ implies
$(2a+2b)(e_1^2)^2(e_1e_2)=0$, therefore $a=-b$, so the multiplicity  of $S_{(4,2)}V$ in $S^3(S^2V)$
is 1.
\end{proof}
\end{example}
\subsection{Brill's equations}\label{polymap}
Following the idea in $\S 8.6$ in \cite{MR2865915}, I use the following notation to define Brill's equations. We first define two maps $\pi_{d,d}$ and $Q_d$, then use them to define Brill's equations.\\\\

Define the projection map $\pi_{d,d}:S^dV\otimes S^dV\rightarrow S_{(d,d)}V$ by
\begin{eqnarray}\label{a}
(l_1\cdots l_d)\otimes(m_1\cdots m_d)\mapsto\sum_{\sigma\in\mathfrak{S}_d}(l_1\wedge m_{\sigma(1)})\cdot(l_2\wedge m_{\sigma(2)})\cdots(l_d\wedge m_{\sigma(d)}),
\end{eqnarray}
and then extend linearly to the whole space.\\\\

Recall $S^\bullet V=\oplus_{i=0}^{\infty}S^iV$. Define a multiplication on $S^\bullet V \otimes S^\bullet V$ by,
for any $a,b,c,d\in S^\bullet V$,
\begin{eqnarray}\label{b}
(a\otimes b)\cdot(c\otimes d)=ac\otimes bd,
\end{eqnarray}
and this extends linearly to  $S^\bullet V \otimes S^\bullet V$.\\

Let $f\in S^\delta V$ and let $f_{j,\delta-j}\in S^jV\otimes S^{\delta-j}V$ be the $j$-th polarization of $f$. Define maps
\begin{eqnarray*}\label{c}
E_j:S^{\delta}V\rightarrow S^jV\otimes S^{j(\delta-1)}V,\\\nonumber
f\mapsto f_{j,\delta-j}\cdot(1\otimes f^{j-1}).
\end{eqnarray*}
If $j>\delta$ define $E_j(f)=0$.\\
\begin{example}
Let $f=l_1l_2l_3\in S^3V$, then
\begin{eqnarray*}\label{c}
E_1(f)&=&f_{1,2}\cdot(1\otimes 1)\\
&=&l_1\otimes l_2l_3+l_3\otimes l_1l_2+l_2\otimes l_1l_3.
\end{eqnarray*}
\begin{eqnarray*}
E_2(f)&=&f_{2,1}\cdot(1\otimes l_1l_2l_3)\\
&=&(l_1l_2\otimes l_3+l_1l_3\otimes l_2+l_2l_3\otimes l_1)\cdot (1\otimes l_1l_2l_3).\\
&=&l_1l_2\otimes l_1l_2l_3^2+l_1l_3\otimes l_1l_2^2l_3+l_2l_3\otimes l_1^2l_2l_3.
\end{eqnarray*}
\begin{eqnarray*}
E_3(f)&=&f_{3,0}\cdot(1\otimes f^2)\\
&=&f\otimes f^2\\
&=&l_1l_2l_3\otimes l_1^2l_2^2l_3^2.
\end{eqnarray*}
\end{example}
The elementary symmetric and power sum function are:
$$e_j=e_j(x_1,\cdots,x_{\underline{v}})=\sum_{1\leq i_1<i_2<\cdots<i_j\leq \underline{v}}x_{i_1}\cdots x_{i_j},$$
$$p_j=p_i(x_1,\cdots,x_{\underline{v}})=\sum_{i=1}^{\underline{v}}x_i^j.$$
The power sum can be written in terms of symmetric function using Girard formula:
\begin{eqnarray}\label{Girard formula}
p_k=\mathcal{P}_k(e_1,\cdots,e_d)=\sum_{i_1+2i_2+\cdots d i_d=k}k(-1)^{k+i_1+i_2+\cdots  i_d}\frac{(i_1+i_2+\cdots  i_d-1)!}{i_1!\cdots i_d!}e_1^{i_1}\cdots e_d^{i_d}.
\end{eqnarray}
\begin{example}
$p_2=\mathcal{P}_2(e_1,e_2)=e_1^2-2e_2$. $p_3=\mathcal{P}_3(e_1,e_2,e_3)=e_1^3-3e_1e_2+3e_3$.
\end{example}
Next, we use Girard formula and $E_j$ to define $Q_d$ .
Define polynomial maps
\begin{eqnarray*}
Q_{d,\delta}: S^\delta V\longrightarrow S^dV\otimes S^{d(\delta-1)}V
\end{eqnarray*}by
\begin{eqnarray}\label{defineQ}
Q_{d,\delta}(f)=\mathcal{P}_d(E_1(f),\cdots,E_d(f)).
\end{eqnarray}
Write $Q_d=Q_{d,d}$.
Explicitly
\begin{equation}\label{Brillexp}
\begin{aligned}
&Q_d(f)=\\
&\sum_{i_1+2i_2+\cdots+di_d=d}d(-1)^{d+i_1+\cdots+i_d}\frac{(i_1+\cdots+i_d-1)!}{i_1!\cdots i_d!} (\prod_{j=1}^{d}f_{j,d-j}^{i_j})\cdot(1\otimes f^{d-(i_1+\cdots+i_d)}).
\end{aligned}
\end{equation}
\begin{example}Let $d=2$, and $f\in S^2V$, by \eqref{Brillexp},
\begin{eqnarray*}
Q_2(f)=f_{1,1}^2-2f\otimes f.
\end{eqnarray*}
\end{example}
\begin{lemma}\label{Qlinear}{\rm ( \S 8.6 \cite{MR2865915})}
Let $l_i\in V$ for $i=1,\cdots,d$, then
\begin{eqnarray}
Q_d(l_1\cdots l_d)=\sum_{j=1}^dl_j^d\otimes(l_1^d\cdots l_{j-1}^dl_{j+1}^d\cdots l_{d}^d).
\end{eqnarray}
\end{lemma}
Now we define Brill's polynomial map $\mathfrak{B}:S^dV\rightarrow S_{(d,d)}V\otimes S^{d^2-d}V$ invariantly. It is the composition of the following
two maps:
$$S^dV\rightarrow S^dV\otimes S^dV\otimes S^{d^2-d}V\rightarrow S_{(d,d)}V\otimes S^{d^2-d}V,$$
where the first map sends $f\in S^dV$ to $f\otimes Q_d(f)$, and the second map is ${\pi_{d,d}\otimes Id_{S^{d^2-d}V}}$.
By Lemma \ref{Qlinear},
\begin{eqnarray*}
\mathfrak{B}(l_1\cdots l_d)&=&{\pi_{d,d}\otimes Id_{S^{d^2-d}V}}[(l_1\cdots l_d)\otimes\sum_{j=1}^dl_j^d\otimes (l_1^d\cdots l_{j-1}^dl_{j+1}^d\cdots l_{d}^d)]\\
&=&0.
\end{eqnarray*}

The converse is also true:

\begin{theorem}\label{brillequations}{\rm (Brill,Gordon \cite{Gordon}, Gelfand-Kapranov-Zelevinsky \cite{gkz}, Briand \cite{MR2664658})}
Consider the polynomial map
\begin{eqnarray*}
\mathfrak{B}:S^dV\rightarrow S_{(d,d)}V\otimes S^{d^2-d}V
\end{eqnarray*}
given by
\begin{eqnarray}
\mathfrak{B}(f)={\pi_{d,d}\otimes Id_{S^{d^2-d}V}}[f\otimes Q_d(f)].
\end{eqnarray}
Then $\mathfrak{B}(f)=0\Leftrightarrow [f]\in Ch_d(V)$.
\end{theorem}
\begin{remark}There was
a gap in Brill's argument, that was repeated in \cite{gkz} and finally fixed by E. Briand in
\cite{MR2664658}.
\end{remark}
\section{The image of Brill's map}\label{imageBrillmap}
\subsection{ Construction of Brill's map}
First  consider the polarization $\overline{Q_d}$ of $Q_d$ , where $Q_d:S^dV\rightarrow S^dV\otimes S^{d^2-d}V$.
\begin{example}Let $d=2$, and $f,g\in S^2V$,  by \eqref{Brillexp}
\begin{eqnarray*}
Q_2(f)=f_{1,1}^2-2f\otimes f.
\end{eqnarray*}
Therefore by \eqref{por}, $\overline{Q_2}:S^2(S^2V)\rightarrow S^2V\otimes S^2V$ is defined by:
\begin{eqnarray*}
\bar{Q}_2(f\cdot g)&=&\frac{1}{2}((f+g)_{1,1}^2-2(f+g)\otimes (f+g)-(f_{1,1}^2-2f\otimes f)-(g_{1,1}^2-2g\otimes g))\\
&=&f_{1,1}g_{1,1}-f\otimes g-g\otimes f.
\end{eqnarray*}
So by \eqref{b}
\begin{eqnarray*}
\bar{Q}_2(e_1e_2\cdot e_1e_2)&=&(e_1e_2)_{1,1}^2-2(e_1e_2)\otimes(e_1e_2)\\
&=&(e_1\otimes e_2+e_2\otimes e_1)^2-2(e_1e_2)\otimes(e_1e_2)\\
&=&e_1^2\otimes e_2^2+e_2^2\otimes e_1^2.
\end{eqnarray*}

\begin{eqnarray*}
\bar{Q}_2(e_1^2\cdot e_1e_2)&=&(e_1e_2)_{1,1}\cdot(e_1^2)_{1,1}-(e_1^2)\otimes(e_1e_2)-(e_1e_2)\otimes(e_1^2)\\
&=&(e_1\otimes e_2+e_2\otimes e_1)\cdot(2e_1\otimes e_1)-(e_1^2)\otimes(e_1e_2)-(e_1e_2)\otimes(e_1^2)\\
&=&e_1^2\otimes e_1e_2+e_1e_2\otimes e_1^2.
\end{eqnarray*}

\begin{eqnarray*}
\bar{Q}_2(e_1e_2\cdot e_1e_3)&=&(e_1e_2)_{1,1}\cdot(e_1e_3)_{1,1}-(e_1e_3)\otimes(e_1e_2)-(e_1e_2)\otimes(e_1e_3)\\
&=&(e_1\otimes e_2+e_2\otimes e_1)\cdot(e_1\otimes e_3+e_3\otimes e_1)-(e_1e_3)\otimes(e_1e_2)-(e_1e_2)\otimes(e_1e_3)\\
&=&e_1^2\otimes e_2e_3+e_2e_3\otimes e_1^2.
\end{eqnarray*}

\begin{eqnarray*}
\bar{Q}_2(e_1e_2\cdot e_3^2)&=&(e_1e_2)_{1,1}\cdot(e_3^2)_{1,1}-(e_3^2)\otimes(e_1e_2)-(e_1e_2)\otimes(e_3^2)\\
&=&(e_1\otimes e_2+e_2\otimes e_1)\cdot(2e_3\otimes e_3)-(e_3^2)\otimes(e_1e_2)-(e_1e_2)\otimes(e_3^2)\\
&=&2e_1e_3\otimes e_2e_3+2e_2e_3\otimes e_1e_3-e_3^2\otimes e_1e_2-e_1e_2\otimes e_3^2.
\end{eqnarray*}
\end{example}

In general, $\overline{Q_d}:S^d(S^dV)\rightarrow S^dV\otimes S^{d^2-d}V$ is used to define Brill's map $\bar{\mathfrak{B}}$:
\begin{lemma}\label{Brillmapexp}
The polarization of Brill's polynomial map $\mathfrak{B}$
\begin{eqnarray*}
\bar{\mathfrak{B}}:S^{d+1}(S^dV)\rightarrow S_{(d,d)}V\otimes S^{d^2-d}V
\end{eqnarray*}
is
\begin{eqnarray}
\bar{\mathfrak{B}}(f_1f_2\ldots f_{d+1})=\frac{1}{d+1}\sum_{i=1}^{d+1}{\pi_{d,d}\otimes Id_{S^{d^2-d}V}}[f_i\otimes\overline{Q_d}(f_1\ldots \hat{f_i}\ldots f_{d+1})].
\end{eqnarray}
\end{lemma}

\begin{example}Consider Brill's map $\bar{\mathfrak{B}}:$$S^{3}(S^2V)\rightarrow S_{(2,2)}V\otimes S^{2}V$ for $d=2$.
By Lemma \ref{Brillmapexp},
\begin{eqnarray*}
\bar{\mathfrak{B}}(e_1e_2\cdot e_1e_2\cdot e_1^2)&=&\frac{1}{3}{\pi_{2,2}\otimes Id_{S^{2}V}}[e_1^2\otimes\overline{Q_2}(e_1e_2\cdot e_1e_2)]\\
&&+\frac{2}{3}{\pi_{2,2}\otimes Id_{S^{2}V}}[e_1e_2\otimes\overline{Q_2}(e_1e_2\cdot e_1^2)]\\
&=&\frac{1}{3}{\pi_{2,2}\otimes Id_{S^{2}V}}[e_1^2\otimes(e_1^2\otimes e_2^2+e_2^2\otimes e_1^2)]\\
&&+\frac{2}{3}{\pi_{2,2}\otimes Id_{S^{2}V}}[e_1e_2\otimes (e_1^2\otimes e_1e_2+e_1e_2\otimes e_1^2)]\\
&=&\frac{1}{3}[2(e_1\wedge e_2)^2\otimes e_1^2)+\frac{2}{3}(-(e_1\wedge e_2)^2\otimes e_1^2]\\
&=&0.
\end{eqnarray*}
\begin{eqnarray*}
\bar{\mathfrak{B}}(e_1e_2\cdot e_1e_2\cdot e_1e_3)&=&\frac{1}{3}{\pi_{2,2}\otimes Id_{S^{2}V}}[e_1e_3\otimes\overline{Q_2}(e_1e_2\cdot e_1e_2)]\\
&&+\frac{2}{3}{\pi_{2,2}\otimes Id_{S^{2}V}}[e_1e_2\otimes\overline{Q_2}(e_1e_2\cdot e_1e_3)]\\
&=&\frac{1}{3}{\pi_{2,2}\otimes Id_{S^{2}V}}[e_1e_3\otimes (e_1^2\otimes e_2^2+e_2^2\otimes e_1^2)]\\
&&+\frac{2}{3}{\pi_{2,2}\otimes Id_{S^{2}V}}[e_1e_2\otimes (e_1^2\otimes e_2e_3+e_2e_3\otimes e_1^2)]\\
&=&\frac{1}{3}[2(e_1\wedge e_2)(e_1\wedge e_3)\otimes e_1^2)]+\frac{2}{3}[-(e_1\wedge e_2)(e_1\wedge e_3)\otimes e_1^2]\\
&=&0.
\end{eqnarray*}
\begin{eqnarray*}
\bar{\mathfrak{B}}(e_1e_2\cdot e_1e_2\cdot e_3^2)&=&\frac{1}{3}{\pi_{2,2}\otimes Id_{S^{2}V}}[e_3^2\otimes\overline{Q_2}(e_1e_2\cdot e_1e_2)]\\
&&+\frac{2}{3}{\pi_{2,2}\otimes Id_{S^{2}V}}[e_1e_2\otimes\overline{Q_2}(e_1e_2\cdot e_3^2)]\\
&=&\frac{1}{3}{\pi_{2,2}\otimes Id_{S^{2}V}}(e_3^2\otimes(e_1^2\otimes e_2^2+e_2^2\otimes e_1^2))\\
&&+\frac{2}{3}{\pi_{2,2}\otimes Id_{S^{2}V}}[e_1e_2\otimes(2e_1e_3\otimes e_2e_3\\
&&+2e_2e_3\otimes e_1e_3-e_3^2\otimes e_1e_2-e_1e_2\otimes e_3^2)]\\
&=&\frac{2}{3}[(e_1\wedge e_3)^2\otimes e_2^2+(e_2\wedge e_3)^2\otimes e_1^2+(e_1\wedge e_2)^2\otimes e_3^2\\
&&-2(e_1\wedge e_2)(e_1\wedge e_3)\otimes e_2e_3-2(e_1\wedge e_2)(e_2\wedge e_3)\otimes e_1e_2\\
&&-2(e_1\wedge e_3)(e_2\wedge e_3)\otimes e_1e_2].
\end{eqnarray*}
\end{example}

\subsection{Brill's map as a $GL(V)$-module map}
Consider Brill's map
$\bar{\mathfrak{B}}:S^{d+1}(S^dV)\rightarrow S_{(d,d)}V\otimes S^{d^2-d}V$.
Recall that the image of Brill's map is isomorphic to dual of the $GL(V)$-module generated by
Brill's equations. Therefore to prove Theorem \ref{Brillmodule}, we only need to prove the following theorem:
\begin{theorem}\label{Brillmap}
Assume $\dim$ $V\geq3$. Consider Brill's map
\begin{eqnarray*}
\bar{\mathfrak{B}}:S^{d+1}(S^dV)\rightarrow S_{(d,d)}V\otimes S^{d^2-d}V.
\end{eqnarray*}
Then
\begin{eqnarray*}
\im(\overline{\mathfrak{B}}) =
\begin{cases}
S_{(7,3,2)}V \ \   d=3;\\
\bigoplus_{j=2}^{d}S_{(d^2-j,d,j)}V\  d\neq3.
\end{cases}
\end{eqnarray*}

\end{theorem}

Brill's map is a $GL(V)$-module map, therefore by Schur's lemma, the image of Brill's map is a $GL(V)$-submodule of $S_{(d,d)}V\otimes S^{d^2-d}V$.
Since we do not know the general decomposition of  $S^{d+1}(S^dV)$, it is impossible to compute the image
of each isotypic component of  $S^{d+1}(S^dV)$ directly. Fortunately, it is easy to decompose the space $S_{(d,d)}V\otimes S^{d^2-d}V$ by Pieri's rule,
i.e.
\begin{eqnarray}
S_{(d,d)}V\otimes S^{d^2-d}V=\bigoplus_{j=0}^{d}S_{(d^2-j,d,j)}V
\end{eqnarray}

Each isotypic component $S_{(d,d)}V\otimes S^{d^2-d}V$ is of multiplicity 1, so the image of Brill's map is multiplicity free.
Also, we only need to consider the modules with length no more than 3, so we only need to consider $V$ to be 3-dimensional from now on.

\subsection{Weight spaces and weight vectors of $S_{(d,d)}V\otimes S^{d^2-d}V$ and $S^{d+1}(S^dV)$}
Let $\{e_1,e_2,e_3\}$ be a basis of V.
\begin{lemma}
As a $GL_3$-module , $S^d(\wedge^2\mathbb{C}^3)$ is $S_{(d,d)}\mathbb{C}^3$.
\end{lemma}
\begin{proof}
First, since $(e_1\wedge e_2)^d\in S^d(\wedge^2\mathbb{C}^3)$ is a highest weight vector with weight $(d,d)$, so $S_{(d,d)}\mathbb{C}^3\subset S^d(\wedge^2\mathbb{C}^3)$. Second,
$\dim\ S_{d,d}\mathbb{C}^3= \dim\ S^d(\wedge^2\mathbb{C}^3)=\binom{d+2}{2}$. The result follows.
\end{proof}
\begin{definition}
Given an integer $j$ such that $j\in\{0,\cdots,d\}$.
Define the weight space $W_j$$\subset$$S^{d+1}(S^dV)$ to be the set of all the degree $d+1$ homogenous polynomials on $S^dV^*$ such that each monomial has weight $(d^2-j,d,j)$ .\\
Define the weight space $\widetilde{W_j}$$\subset$$S_{(d,d)}V\otimes S^{d^2-d}V=S^d(\wedge^2V)\otimes S^{d^2-d}V$ to be the set of all the weight vectors in $S^d(\wedge^2V)\otimes S^{d^2-d}V$ whose weights are $(d^2-j,d,j)$.
\end{definition}
\begin{lemma}\label{basis}
The weight space $\widetilde{W_j}$$\subset$$S_{(d,d)}V\otimes S^{d^2-d}V= S^d(\wedge^2V)\otimes S^{d^2-d}V$ has indeed
basis$$\{(e_1\wedge e_2)^{d+s-j-t}(e_1\wedge e_3)^t(e_2\wedge e_3)^{j-s}\otimes e_1^{d^2-d-s}e_2^te_3^{s-t}\}_{0\leq s \leq j,0\leq t \leq s}.$$
\end{lemma}
\begin{proof}
$S^d(\wedge^2V)\otimes S^{d^2-d}V$ has a indeed basis $$\{(e_1\wedge e_2)^{d-a_1-a_2}(e_1\wedge e_3)^{a_1}(e_2\wedge e_3)^{a_2}\otimes e_1^{d^2-d-a_3-a_4}e_2^{a_3}e_3^{a_4}\}_{0\leq a_1+a_2\leq d,0\leq a_3+a_4 \leq d^2-d}.$$ Let $v\in W_j$ be a basis vector of $S^d(\wedge^2V)\otimes S^{d^2-d}V$. Then
\begin{eqnarray}
\left\{
\begin{aligned}
a_1+a_2+a_4=0\\
a_1-a_3=0.\\
\end{aligned}
\right.
\end{eqnarray}
Let $a_3=t, a_3+a_4=s$, then $0\leq s \leq j,0\leq t \leq s$ and $v=(e_1\wedge e_2)^{d+s-j-t}(e_1\wedge e_3)^t(e_2\wedge e_3)^{j-s}\otimes e_1^{d^2-d-s}e_2^te_3^{s-t}$.
\end{proof}
\begin{lemma}\label{hwvectors}The highest weight vector $\tilde{v_j}\in S_{(d^2-j,d,j)}V\subset S_{(d,d)}V\otimes S^{d^2-d}V= S^d(\wedge^2V)\otimes S^{d^2-d}V$ is
\begin{eqnarray}
\sum_{s=0}^{j}\sum_{t=0}^{s}(-1)^t\binom{j}{s}\binom{s}{t}(e_1\wedge e_2)^{d+s-j-t}(e_1\wedge e_3)^{t}(e_2\wedge e_3)^{j-s}\otimes e_1^{d^2-d-s}e_2^te_3^{s-t}.
\end{eqnarray}
\begin{proof}
By Lemma \ref{basis}, write
\begin{eqnarray*}
\tilde{v_j}=\sum_{s=0}^{j}\sum_{t=0}^{s}a_{st}(e_1\wedge e_2)^{d+s-j-t}(e_1\wedge e_3)^{t}(e_1\wedge e_2)^{j-s}\otimes e_1^{d^2-d-s}e_2^te_3^{s-t}.
 \end{eqnarray*}
 Apply raising operators $E_{2}^{1}$ and $ E_3^2$ on $\tilde{v_j}$,
 \begin{eqnarray*}
E_{2}^{1}\tilde{v_j}&=&\sum_{s=0}^{j}\sum_{t=0}^{s}a_{st}(j-s)(e_1\wedge e_2)^{d+s-j-t}(e_1\wedge e_3)^{t+1}(e_1\wedge e_2)^{j-s-1}\otimes e_1^{d^2-d-s}e_2^te_3^{s-t}.\\
&&+\sum_{s=0}^{j}\sum_{t=0}^{s}ta_{st}(e_1\wedge e_2)^{d+s-j-t}(e_1\wedge e_3)^{t}(e_2\wedge e_3)^{j-s-1}\otimes e_1^{d^2-d-s+1}e_2^{t-1}e_3^{s-t}\\
&=&\sum_{s=0}^{j-1}\sum_{t=1}^{s+1}(ta_{s+1,t}+(j-s)a_{s,t-1})(e_1\wedge e_2)^{d+s-j-t}(e_1\wedge e_3)^{t}(e_2\wedge e_3)^{j-s-1}\\
&&\otimes e_1^{d^2-d-s}e_2^{t-1}e_3^{s-t}.\\
\end{eqnarray*}
and
 \begin{eqnarray*}
E_{3}^{2}\tilde{v_j}&=&\sum_{s=0}^{j}\sum_{t=0}^{s}ta_{st}(e_1\wedge e_2)^{d+s-j-t+1}(e_1\wedge e_3)^{t-1}(e_1\wedge e_2)^{j-s}\otimes e_1^{d^2-d-s}e_2^te_3^{s-t}\\
&&+\sum_{s=0}^{j}\sum_{t=0}^{s}(s-t)a_{st}(e_1\wedge e_2)^{d+s-j-t}(e_1\wedge e_3)^{t}(e_2\wedge e_3)^{j-s}\otimes e_1^{d^2-d-s}e_2^{t-+}e_3^{s-t-1}\\
&=&\sum_{s=1}^{j}\sum_{t=1}^{s}(ta_{s,t}+(s-t+1)a_{s,t-1})(e_1\wedge e_2)^{d+s-j-t+1}(e_1\wedge e_3)^{t-1}(e_2\wedge e_3)^{j-s}\\
&&\otimes e_1^{d^2-d-s}e_2^{t}e_3^{s-t}.\\
\end{eqnarray*}
 we get two systems of equations of $\{a_{st}\}_{0\leq s \leq j,0\leq t \leq s}$,
\begin{eqnarray}
\left\{
\begin{aligned}
ta_{s+1,t}+(j-s)a_{s,t-1}=0\\
ta_{s,t}+(s-t+1)a_{s,t-1}=0
\end{aligned}
\right.
\end{eqnarray}
 and then solve for $\{a_{st}\}_{0\leq s \leq j,0\leq t \leq s}$, we get a unique solution $a_{s,t}=(-1)^t\binom{j}{s}\binom{s}{t}$ up to scale.
\end{proof}
\end{lemma}
Since Brill's map is a $GL(V)$-module map, we only need to check whether $\tilde{v_j}$ is in the image of Brill's map.

For convenience, write
\begin{eqnarray}
S^{d+1}(S^dV)=A_d \bigoplus(\bigoplus_{j=0}^{d}{S_{(d^2-j,d,j)}V}^{\oplus m_j}).
\end{eqnarray}
Where $A_d$ is the direct sum of the isotypic components of $S^{d+1}(S^dV)$ other than $S_{(d^2-j,d,j)}V$ for $j=0,1,\cdots,d$,
which is certainly in the kernel of Brill's map.\\

The idea is to take $v_j=(e_1^{
d-1}e_2)^d(e_1^{d-j}{e_3}^j)\in W_j$ , compute $\bar{\mathfrak{B}}(v_j)$, and see whether the projection of
$\bar{\mathfrak{B}}(v_j)$ to $S_{(d^2-j,d,j)}V\subset S_{(d,d)}V\otimes S^{d^2-d}V$ is 0.

\begin{proposition}\label{test1}
If the projection of $\bar{\mathfrak{B}}(v_j)$ to $S_{(d^2-j,d,j)}V\subset S_{(d,d)}V\otimes S^{d^2-d}V$ is not 0, then
$\tilde{v_j}$ is in the image of Brill's map, therefore $S_{(d^2-j,d,j)}V\subset S_{(d,d)}V\otimes S^{d^2-d}V$ is
in the image of Brill's map.
\end{proposition}
\begin{proof}
Write $v_j=v_{j1}+v_{j2}+v_{j3}$, where $v_{j1}\in A_d$, $v_{j2}\in \bigoplus_{k=0}^{j-1}{S_{(d^2-k,d,k)}V}^{\oplus m_k}$ ,
and $v_{j3}\in {S_{(d^2-j,d,j)}V}^{\oplus m_j}$ is a highest weight vector. By Schur's Lemma, $\bar{\mathfrak{B}}(v_{j1})=0$,
$\bar{\mathfrak{B}}(v_{j2})\in \bigoplus_{k=0}^{j-1}S_{(d^2-k,d,k)}V$, and $\bar{\mathfrak{B}}(v_{j3})\in S_{(d^2-j,d,j)}V$,
therefore the projection of
$\bar{\mathfrak{B}}(v_j)$ to $S_{(d^2-j,d,j)}V\subset S_{(d,d)}V\otimes S^{d^2-d}V$  is exactly $\bar{\mathfrak{B}}(v_{j3})$,
by Schur's Lemma, if it is not 0, it is ${\tilde{v_j}}$ (see Lemma \ref{hwvectors}) up to a constant.
\end{proof}
\subsection{Computing $\bar{\mathfrak{B}}(v_j)$ }
Brill's map is very complicated to compute in general. Fortunately, we are able to compute $\bar{\mathfrak{B}}(v_j)$.
\begin{eqnarray*}
\bar{\mathfrak{B}}(v_j)&=&\bar{\mathfrak{B}}((e_1^{d-1}e_2)^d\cdot(e_1^{d-j}{e_3}^j))\\
&=&\frac{1}{d+1}{\pi_{d,d}\otimes Id_{S^{d^2-d}V}}((e_1^{d-j}{e_3}^j)\otimes\overline{Q_d}((e_1^{d-1}e_2)^d)\\
&&+\frac{d}{d+1}{\pi_{d,d}\otimes Id_{S^{d^2-d}V}}((e_1^{d-1}e_2)\otimes\overline{Q_d}((e_1^{d-1}e_2)^{d-1}\cdot(e_1^{d-j}{e_3}^j))\\
&=&\frac{1}{d+1}{\pi_{d,d}\otimes Id_{S^{d^2-d}V}}((e_1^{d-j}{e_3}^j)\otimes Q_d(e_1^{d-1}e_2))\\
&&+\frac{d}{d+1}{\pi_{d,d}\otimes Id_{S^{d^2-d}V}}((e_1^{d-1}e_2)\otimes\overline{Q_d}((e_1^{d-1}e_2)^{d-1}\cdot(e_1^{d-j}{e_3}^j)))
\end{eqnarray*}
\\\\
First, I compute and ${\pi_{d,d}\otimes Id_{S^{d^2-d}V}}((e_1^{d-j}{e_3}^j)\otimes Q_d(e_1^{d-1}e_2))$.
By Lemma \ref{Qlinear},
\begin{proposition}\label{firstpart} ${\pi_{d,d}\otimes Id_{S^{d^2-d}V}}((e_1^{d-j}{e_3}^j)\otimes Q_d(e_1^{d-1}e_2))$ is
\begin{eqnarray}
\left\{
\begin{aligned}
&{d!}(e_1\wedge e_2)^{d-j}(e_3\wedge e_2)^{j}\otimes e_1^{d^2-d}\  \ j\neq d\\
&{d!}(e_3\wedge e_2)^{d}\otimes(e_1^{d^2-d})+(d-1)d!\otimes(e_3\wedge e_1)^d\otimes e_1^{d^2-2d}e_2^d\ \ j=d
\end{aligned}
\right.
\end{eqnarray}
\end{proposition}

Next, I compute ${\pi_{d,d}\otimes Id_{S^{d^2-d}V}}((e_1^{d-1}e_2)\otimes\overline{Q_d}((e_1^{d-1}e_2)^{d-1}(e_1^{d-j}{e_3}^j)))$.
\begin{lemma}\label{divisible}
If $h\in S^dV$ is divisible by $e_1^2$,  then  $\pi_{d,d}(h,e_1^{d-1}e_2)=0$.
\end{lemma}
\begin{lemma}\label{Brillpor}
For any $f,g\in S^dV$, by polarizing \eqref{Brillexp},
\begin{eqnarray*}
\overline{Q_d}(f^{d-1}g)&=&(-1)^d\sum_{i_1+2i_2+\cdots+di_d=d}d(-1)^{i_1+i_2+\cdots+i_d}\frac{(i_1+i_2+\cdots+i_d-1)!}{i_1!\cdots i_d!}\\
& &(\sum_{s=1}^d\frac{i_s}{d}(\prod_{j=1,j\neq s}^{d}f_{j,d-j}^{i_j})\cdot(f_{s,d-s}^{i_s-1}g_{s,d-s})\cdot(1\otimes f^{d-(i_1+\cdots+i_d)}))\\
&&+(\frac{d-(i_1+\cdots+i_d)}{d}(\prod_{j=1}^{d}f_{j,d-j}^{i_j})\cdot(1\otimes f^{d-(i_1+\cdots+i_d)-1}g)).
\end{eqnarray*}
\end{lemma}

Now I use Lemma \ref{Brillpor} to compute $\overline{Q_d}((e_1^{d-1}e_2)^{d-1}\cdot(e_1^{d-j}{e_3}^j))$. By lemma \ref{divisible}, terms of  $\overline{Q_d}((e_1^{d-1}e_2)^{d-1}\cdot(e_1^{d-j}{e_3}^j))$ whose first components are
divisible by $e_1^2$ are killed by $(e_1^{d-1}e_2)$ via $\pi_{d,d}$. Therefore,
by Lemma \ref{Brillpor}, given $i_1,\cdots,i_d$ with $i_1+2i_2+\cdots d i_d=d$, we need $\#\{j\geq3|i_j\geq1\}\leq 1$ so that the corresponding terms will not vanish. There are 2 possibilities, either some $i_s=1$ for some $s\geq3$
or $i_s\equiv0$ for all $s\geq3 $. More specifically, there are five cases for which $\overline{Q_d}((e_1^{d-1}e_2)^{d-1}(e_1^{d-j}{e_3}^j))$ may not vanish:
\begin{enumerate}
\item $i_s=1$ for some $s\geq3$ and $ i_2=0, i_1=d-s$;
\item $i_s=1$ for some $s\geq3$ and $ i_2=1, i_1=d-s-2$;
\item $i_s\equiv0$ for all $s\geq3$ and  $ i_2=0, i_1=d$;
\item $i_s\equiv0$ for all $s\geq3$ and $ i_2=1, i_1=d-2$;
\item $i_s\equiv0$ for all $s\geq3$ and $ i_2=2, i_1=d-4$.
\end{enumerate}

I use the symbol $\equiv$ to omit those terms of  $\overline{Q_d}((e_1^{d-1}e_2)^{d-1}\cdot(e_1^{d-j}{e_3}^j))$ whose first components
are divisible by $e_1^2$. I use $I_1$ to denote the terms of the first case in $\overline{Q_d}((e_1^{d-1}e_2)^{d-1}(e_1^{d-j}{e_3}^j))$.
For the first case,  $i_s=1$ for some $s\geq3$ and $ i_2=0, i_1=d-s$,
so the coefficient of the terms of the first case is
\begin{eqnarray*}
(-1)^d(-1)^{i_1+i_2+\cdots  i_d}\frac{(i_1+i_2+\cdots  i_d-1)!}{i_1!\cdots i_d!}=(-1)^dd(-1)^{d-s-1}\frac{(d-s)!}{(d-s)!}=d(-1)^{s-1},
\end{eqnarray*}
and the corresponding monomial in $ Q_d(f)$ is
\begin{eqnarray*}d(-1)^{s-1}f_{1,d-1}^{d-s}f_{s,d-s}\cdot(1\otimes f^{s-1}).
\end{eqnarray*}
 Since the first component of $(e_1^{d-1}e_2)_{s,d-s}$ is divisible by $e_1^2$, by lemma \ref{divisible}, in order that the terms will not be killed by $(e_1^{d-1}e_2)$ via $\pi_{d,d}$, $e_1^{d-j}{e_3}^j$ should replace $(e_1^{d-1}e_2)$ in the position $f_{s,d-s}$.
By Lemma \ref{Brillpor},
\begin{eqnarray*}
I_1&\equiv&\sum_{s=3}^dd(-1)^{s-1}\frac{1}{d}(e_1^{d-1}e_2)_{1,d-1}^{d-s}
(e_1^{d-j}{e_3}^j)_{s,d-s}\cdot[1\otimes (e_1^{d-1}e_2)^{s-1}]\\
&=&\sum_{s=3}^{d}(-1)^{s-1}(e_1^{d-1}e_2)_{1,d-1}^{d-s}(e_1^{d-j}{e_3}^j)_{s,d-s}\cdot[1\otimes (e_1^{d-1}e_2)^{s-1}]\\
&\equiv&\sum_{s=3}^{d}(-1)^{s-1}[(d-1)e_1\otimes e_1^{d-2}e_2+e_2\otimes e_1^{d-1}]^{d-s}[\binom{j}{s}e_3^s\otimes e_1^{d-j}e_3^{j-s}\\
&&+(d-j)\binom{j}{s-1}e_1e_3^{s-1}\otimes e_1^{d-j-1}e_3^{j-s+1}]\cdot[1\otimes (e_1^{d-1}e_2)^{s-1}]\\
&\equiv&\sum_{s=3}^{d}(-1)^{s-1}[(d-1)(d-s)e_1e_2^{d-s-1}\otimes e_1^{d-2}e_2e_1^{(d-1)(d-s-1)}+e_2^{d-s}\otimes e_1^{(d-1)(d-s)}]\\
&&[\binom{j}{s}e_3^s\otimes e_1^{d-j}e_3^{j-s} +(d-j)\binom{j}{s-1}e_1e_3^{s-1}\otimes e_1^{d-j-1}e_3^{j-s+1}]\cdot[1\otimes (e_1^{d-1}e_2)^{s-1}]\\
&\equiv&\sum_{s=3}^{d}(-1)^{s-1}[\binom{j}{s}e_2^{d-s}e_3^s\otimes e_1^{d^2-j-d+1}e_2^{s-1}e_3^{j-s}\\
&&+(d-j)\binom{j}{s-1}e_1e_2^{d-s}e_3^{s-1}\otimes e_1^{d^2-j-d}e_2^se_3^{j-s+1}\\
&&+(d-1)(d-s)\binom{j}{s}e_1e_2^{d-s-1}e_3^s\otimes e_1^{d^2-j-d}e_2^{s}e_3^{j-s}].
\end{eqnarray*}

Similarly for the other four cases,
\begin{eqnarray*}
I_2&\equiv&\sum_{s=3}^{d}d(-1)^s(d-s-1)\frac{1}{d}(e_1^{d-1}e_2)_{1,d-1}^{d-s-2}(e_1^{d-1}e_2)_{2,d-2}(e_1^{d-j}{e_3}^j)_{s,d-s}\cdot[1\otimes (e_1^{d-1}e_2)^{s}]\\
&\equiv&\sum_{s=3}^{d}(-1)^s(d-s-1)[(d-1)e_1\otimes e_1^{d-2}e_2+e_2\otimes e_1^{d-1}]^{d-s-2}((d-1)e_1e_2\otimes e_1^{d-2})\\
&&[\binom{j}{s}e_3^s\otimes e_1^{d-j}e_3^{j-s}+(d-j)\binom{j}{s-1}e_1e_3^{s-1}\otimes e_1^{d-j-1}e_3^{j-s+1}]\cdot[1\otimes (e_1^{d-1}e_2)^{s-1}]\\
&\equiv&\sum_{s=3}^{d}(-1)^s(d-s-1)(d-1)\binom{j}{s}e_1e_2^{d-s-1}e_3^s\otimes e_1^{d^2-d-j}e_2^{s}e_3^{j-s}.
\end{eqnarray*}

\begin{eqnarray*}
I_3&\equiv&\frac{d(d-3)}{2}\frac{2}{d}(e_1^{d-1}e_2)_{1,d-1}^{d-4}(e_1^{d-1}e_2)_{2,d-2}(e_1^{d-j}{e_3}^j)_{2,d-2}\cdot[1\otimes(e_1^{d-1}e_2)^{2}]\\
&\equiv&(d-3)[(d-1)e_1\otimes e_1^{d-2}e_2+e_2\otimes e_1^{d-1}]^{d-4}[(d-1)e_1e_2\otimes e_1^{d-2}]\\
&&[\binom{j}{2}e_3^2\otimes e_1^{d-j}e_3^{j-2}]\cdot[1\otimes (e_1^{d-1}e_2)^{2}]\\
&\equiv&(d-3)(d-1)\binom{j}{s}e_1e_2^{d-3}e_3^2\otimes e_1^{d^2-d-j}e_2^{2}e_3^{j-2}.
\end{eqnarray*}

\begin{eqnarray*}
I_4&\equiv&-(d-2)(e_1^{d-1}e_2)_{1,d-1}^{d-3}(e_1^{d-j}{e_3}^j)_{1,d-1}(e_1^{d-1}e_2)_{2,d-2}\cdot[1\otimes (e_1^{d-1}e_2)]\\
&&-(e_1^{d-1}e_2)_{1,d-1}^{d-2}(e_1^{d-j}{e_3}^j)_{2,d-2}\cdot(1\otimes (e_1^{d-1}e_2))-(e_1^{d-1}e_2)_{1,d-1}^{d-2}(e_1^{d-1}e_2)_{2,d-2}\cdot[1\otimes (e_1^{d-j}{e_3}^j)]\\
&\equiv&-(d-2)[(d-1)e_1\otimes e_1^{d-2}e_2+e_2\otimes e_1^{d-1}]^{d-3}(j e_3\otimes e_1^{d-j}e_3^{j-1})[(d-1)e_1e_2\otimes e_1^{d-2}]\cdot[1\otimes (e_1^{d-1}e_2)]\\&&-[(d-1)e_1\otimes e_1^{d-2}e_2+e_2\otimes e_1^{d-1}]^{d-2}[(d-j)j e_1e_3\otimes e_1^{d-j-1}e_3^{j-1}
+\binom{j}{2}e_3^2\otimes e_1^{d-j}e_3^{j-2}]\\&&
\cdot[1\otimes (e_1^{d-1}e_2)]-[(d-1)e_1\otimes e_1^{d-2}e_2+e_2\otimes e_1^{d-1}]^{d-2}[(d-1)e_1e_2\otimes e_1^{d-2}]\cdot[1\otimes (e_1^{d-j}{e_3}^j)]\\
&\equiv&-(d-2)(e_2^{d-3}\otimes e_1^{(d-1)(d-3)})(j e_3\otimes e_1^{d-j}e_3^{j-1})((d-1)e_1e_2\otimes e_1^{d-2})\cdot(1\otimes (e_1^{d-1}e_2))\\
&&-[e_2^{d-2}\otimes e_1^{(d-1)(d-2)}+(d-1)(d-2)e_1e_2^{d-2}\otimes e_1^{d^2-3d+1}e_2][(d-j)j e_1e_3\otimes e_1^{d-j-1}e_3^{j-1}\\&&+\binom{j}{2}e_3^2\otimes e_1^{d-j}e_3^{j-2}]\cdot[1\otimes (e_1^{d-1}e_2)]
-(e_2^{d-2}\otimes e_1^{(d-1)(d-2)})[(d-1)e_1e_2\otimes e_1^{d-2}]\cdot(1\otimes [e_1^{d-j}{e_3}^j)]\\
&\equiv&[-(d-2)(d-1)j-j(d-j)] e_1e_2^{d-2}e_3\otimes e_1^{d^2-d-j}e_2e_3^{j-1}-\binom{j}{2}e_2^{d-2}e_3^2\otimes e_1^{d^2-j-d+1}e_2e_3^{j-2}\\&&-(d-2)(d-1)\binom{j}{2}e_1e_2^{d-3}e_3^2\otimes e_1^{d^2-j-d}e_2^2e_3^{j-2}-
(d-1)e_1e_2^{d-1}\otimes e_1^{d^2-j-d}e_3^{j}.
\end{eqnarray*}
\begin{eqnarray*}
I_5&\equiv&(e_1^{d-1}e_2)_{1,d-1}^{d-1}(e_1^{d-j}{e_3}^j)_{1,d-1}\\
&\equiv&[(d-1)e_1\otimes e_1^{d-2}e_2+e_2\otimes e_1^{d-1})]^{d-1}[j e_3\otimes e_1^{d-j}e_3^{j-1}+(d-j)e_1\otimes e_3^je_1^d]\\
&\equiv&[e_2^{d-1}\otimes e_1^{(d-1)(d-1)}+(d-1)^2e_1^{d-2}e_2\otimes e_1^{(d-2)^2+d-1}e_2^{d-2}][j e_3\otimes e_1^{d-j}e_3^{j-1}+(d-j)e_1\otimes e_3^je_1^d]\\
&\equiv&j e_2^{d-1}e_3\otimes e_1^{d^2-j}e_2e_3^{j-1}+(d-j)e_1e_2^{d-1}\otimes e_1^{d^2-j-d}e_3^{j}+j(d-1)^2e_1e_2^{d-2}e_3\otimes e_1^{d^2-j-d}e_2e_3^{j-1}.
\end{eqnarray*}
Therefore
\begin{eqnarray*}
\overline{Q_d}((e_1^{d-1}e_2)^{d-1}(e_1^{d-j}{e_3}^j))&\equiv&\sum_{s=0}^{min\{j,d-1\}}(-1)^s\binom{j}{s}(1-j)e_1e_2^{d-s-1}e_3^s\otimes e_1^{d^2-d-j}e_2^{s}e_3^{j-s}\\
&&+\sum_{s=1}^{j}(-1)^{s-1}\binom{j}{s}e_2^{d-s}e_3^s\otimes e_1^{d^2-d-j+1}e_2^{s-1}e_3^{j-s}.
\end{eqnarray*}
This implies
\begin{proposition}\label{secondpart}
\begin{eqnarray*}
&&{\pi_{d,d}\otimes Id_{S^{d^2-d}V}}((e_1^{d-1}e_2)\otimes\overline{Q_d}((e_1^{d-1}e_2)^{d-1}\cdot(e_1^{d-j}{e_3}^j))=\\
&&\sum_{s=0}^{min\{j,d-1\}}(-1)^{s-1}\binom{j}{s}(1-j)(d-1)!(e_1\wedge e_2)^{d-s}(e_1\wedge e_3)^{s}\otimes e_1^{d^2-d-j}e_2^{s}e_3^{j-s}\\
&&+\sum_{s=1}^{j}(-1)^{s-1}\binom{j}{s}s(d-1)!(e_1\wedge e_2)^{d-s}(e_1\wedge e_3)^{s-1}(e_2\wedge e_3)\otimes e_1^{d^2-d-j+1}e_2^{s-1}e_3^{j-s}.
\end{eqnarray*}
\end{proposition}

 Proposition \ref{firstpart} and Proposition \ref{secondpart} imply:
\begin{proposition}\label{Brillimage}
\begin{eqnarray*}
&&\bar{\mathfrak{B}}((e_1^{d-1}e_2)^d\cdot(e_1^{d-j}{e_3}^j))=\frac{d!}{d+1}(e_1\wedge e_2)^{d-j}(e_3\wedge e_2)^{j}\otimes(e_1^{d^2-d})\\
&&+\frac{d!}{d+1}\sum_{s=0}^{j}(-1)^{s-1}\binom{j}{s}(1-j)(e_1\wedge e_2)^{d-s}(e_1\wedge e_3)^{s}\otimes e_1^{d^2-d-j}e_2^{s}e_3^{j-s}\\
&&+\frac{d!}{d+1}\sum_{s=1}^{j}(-1)^{s-1}\binom{j}{s}s(e_1\wedge e_2)^{d-s}(e_1\wedge e_3)^{s-1}(e_2\wedge e_3)\otimes e_1^{d^2-d-j+1}e_2^{s-1}e_3^{j-s}.
\end{eqnarray*}
\end{proposition}

\subsection{Orthogonal decomposition of $S_{(d,d)V}\otimes S^{d^2-d}V$}
Let ${e_1,e_2,e_3}$ be a basis of V and define a Hermitian inner product on V such that
\begin{eqnarray*}
<e_i,e_j>=\delta_{i,j}.
\end{eqnarray*}
Extend the Hermitian inner product to $V^{\otimes(d^2+d)}$ naturally by
\begin{eqnarray*}
<e_{i_1}\otimes\cdots\otimes e_{i_{d^2+d}},e_{j_1}\otimes\cdots \otimes e_{j_{d^2+d}}>=\delta_{i_1,j_1}\cdots \delta_{i_{d^2+d},j_{d^2+d}}.
\end{eqnarray*}
One can decompose $V^{\otimes(d^2+d)}$ into direct sum of isotypic components as a $GL(V)$-module.
Since the Hermitian inner product is unitary invariant, distinct isotypic components of $V^{\otimes(d^2+d)}$
are orthogonal (see e.g. \cite{FH}).

Consider $S_{(d,d)}V\otimes S^{d^2-d}V= S^d(\wedge^2V)\otimes S^{d^2-d}V $ as a subspace of $V^{\otimes(d^2+d)} $,
the decomposition $S_{(d,d)}V\otimes S^{d^2-d}V=\bigoplus_{j=0}^{d}S_{(d^2-j,d,j)}V$ is an orthogonal decomposition
with respect to the Hermitian inner product,
therefore
\begin{proposition}\label{test2}
 The projection of $\bar{\mathfrak{B}}(v_j)$ on $S_{(d^2-j,d,j)}V\subset S_{(d,d)}V\otimes S^{d^2-d}V$ is not 0 if and only if $<B(v_j),\tilde{v_j}>\neq 0$, where $\tilde{v_j}$ is defined in Lemma \ref{hwvectors}.
\end{proposition}
\begin{lemma}\label{innerproduct}
Suppose $a_1+a_2+a_3=d$ and $b_1+b_2+b_3=d^2-d$,
\begin{eqnarray*}
& &<(e_1\wedge e_2)^{a_1}(e_1\wedge e_3)^{a_2}(e_2\wedge e_3)^{a_3}\otimes e_1^{b_1}e_2^{b_2}e_3^{b_3},(e_1\wedge e_2)^{a_1}(e_1\wedge e_3)^{a_2}(e_2\wedge e_3)^{a_3}\otimes e_1^{b_1}e_2^{b_2}e_3^{b_3}>\\&=&(\frac{1}{2})^d\frac{a_1!a_2!a_3!}{d!}\frac{b_1!b_2!b_3!}{(d^2-d)!}
\end{eqnarray*}
\end{lemma}
Recall by Lemma \ref{hwvectors},
\begin{eqnarray*}
\tilde{v_j}=\sum_{s=0}^{j}\sum_{t=0}^{s}(-1)^t\binom{j}{s}\binom{s}{t}(e_1\wedge e_2)^{d+s-j-t}(e_1\wedge e_3)^{t}(e_1\wedge e_2)^{j-s}\otimes e_1^{d^2-d-s}e_2^te_3^{s-t}.
 \end{eqnarray*}
 and by Proposition \ref{Brillimage},
 \begin{eqnarray*}
&&\bar{\mathfrak{B}}((e_1^{d-1}e_2)^d\cdot(e_1^{d-j}{e_3}^j))=\frac{d!}{d+1}(-1)^j(e_1\wedge e_2)^{d-j}(e_2\wedge e_3)^{j}\otimes(e_1^{d^2-d})\\
&&+\frac{d!}{d+1}\sum_{t=0}^{j}(-1)^{t}\binom{j}{t}(j-1)(e_1\wedge e_2)^{d-t}(e_1\wedge e_3)^{t}\otimes e_1^{d^2-d-j}e_2^{t}e_3^{j-t}\\
&&+\frac{d!}{d+1}\sum_{t=1}^{j}(-1)^{t-1}\binom{j}{t}t(e_1\wedge e_2)^{d-t}(e_1\wedge e_3)^{t-1}(e_2\wedge e_3)\otimes e_1^{d^2-d-j+1}e_2^{t-1}e_3^{j-t}.
\end{eqnarray*}
By Lemma \ref{innerproduct},
\begin{proposition}
For any fixed $j\in \{0,1,\cdots,d\}$,
\begin{eqnarray*}
<B(v_j),\tilde{v_j}>&=&\frac{(\frac{1}{2})^d}{(d+1)(d^2-d)!}(\sum_{t=0}^{j}\frac{(j!)^2(j-1)(d-t)!(d^2-d-j)!}{(j-t)!}\\
&+&\sum_{t=0}^{j-1}\frac{(j!)^2(d-t-1)!(d^2-d-j+1)!}{(j-t-1)!}+(-1)^j(d-j)!j!(d^2-d)!).
\end{eqnarray*}
$<B(v_j),\tilde{v_j}>=0$ only when
\begin{enumerate}
\item $j=0,1$ for all $d\geq2$;
\item $j=3$ and $d=3$.
\end{enumerate}
\end{proposition}
\begin{proof}
The ratio of $(d-j)!j!(d^2-d)!$ and $\frac{(j!)^2(j-1)(d-t)!(d^2-d-j)!}{(j-t)!}$ is
$\frac{\binom{d^2-d}{j}}{(j-1)\binom{d-t}{j-t}}$, and the ratio of $(d-j)!j!(d^2-d)!$ and $\frac{(j!)^2(d-t-1)!(d^2-d-j+1)!}{(j-t-1)!}$ is
$\frac{\binom{d^2-d}{j-1}}{j\binom{d-t-1}{j-t-1}}$. Therefore when $d$ is large enough and $j\geq2$, the term $(-1)^j(d-j)!j!(d^2-d)!$ dominates.
For small cases, one can check directly.
\end{proof}
Combining all the results above, we  prove Theorem \ref{Brillmap} to prove Theorem \ref{Brillmodule}.

\begin{proof}[Proof of Theorem \ref{Brillmap}]
First, for $j=0$ and $d\geq2$, $S_{(d^2,d)}V$ is not in the image of Brill's map because
$Ch_d(\mathbb{C}^2)=S^d\mathbb{C}^2$.

Second, for $j=1$ and all $d\geq2$, $S_{(d^2-1,d,1)}V$ is not in the image of Brill's map.
If it were in the image of Brill's map, then $\mathfrak{B}(v_1)=\mathfrak{B}(v_{13})=C\tilde{v_1}\neq0$ (where $v_{13}$ is defined in the proof of Proposition \ref{test1}),
so $<B(v_1),\tilde{v_1}>=<C\tilde{v_1},\tilde{v_1}>\neq0$, contradiction.

Third, when $d=3$ and $j=3$, the module $S_{(6,3,3)}V$ is not in the decomposition of $S^4(S^3V)$,
so it is not in the image.

Finally, for other cases, $<B(v_j),\tilde{v_j}>\neq0$, by Proposition \ref{test1} and Proposition \ref{test2},
$S_{(d^2-j,d,j)}V$ is in the image of Brill's map.
\end{proof}

\bibliographystyle{amsplain}
\bibliography{Lmatrix}
\end{document}